\documentclass[12pt]{article}
\usepackage{latexsym}
\usepackage{amsmath, amsthm, amsfonts}
\usepackage{graphics}
\usepackage{graphicx}
\usepackage{color}
\usepackage[table,xcdraw]{xcolor}
\usepackage{float}
\usepackage{array}
\usepackage{hyperref}

\newcommand{\abs}[1]{\left\lvert#1\right\rvert}
\newtheorem{theorem}{Theorem}[section]

\newtheorem{openproblem}[theorem]{Open Problem}

\def\smallskip{\addvspace{\smallskipamount}}
\def\medskip{\addvspace{\medskipamount}}
\def\bigskip{\addvspace{\bigskipamount}}
\def\makefootline{\baselineskip=24pt \line{\the\footline}}
\def\pagecontents{\ifvoid\topins\else\unvbox\topins\fi
   \dimen@=\dp255 \unvbox255
   \ifvoid\footins\else
      \vskip\skip\footins \footnoterule \unvbox\footins\fi
     \ifr@ggedbottom \kern-\dimen@ \vfil \fi}
\def\footnoterule{\kern-3pt\hrule width 2truein \kern 2.6pt}

\begin{document}

\title{{Dynamics of Delay Logistic Difference Equation in the Complex Plane}}

\author{Sk. Sarif Hassan\\
  \small {International Centre for Theoretical Sciences}\\
  \small {Tata Institute of Fundamental Research}\\
  \small {Bangalore $560012$, India}\\
  \small Email: {\texttt{sarif.hassan@icts.res.in}}\\
}

\maketitle
\begin{abstract}
\noindent The dynamics of the delay logistic equation with complex parameters and arbitrary complex initial conditions is investigated. The analysis of the local stability of this difference equation has been carried out. We further exhibit several interesting characteristics of the solutions of this equation, using computations, which does not arise when we consider the same equation with positive real parameters and initial conditions. Some of the interesting observations led us to pose some open problems and conjectures regarding chaotic and higher order periodic solutions and global asymptotic convergence of the delay logistic equation. It is our hope that these observations of this complex difference equation would certainly be an interesting addition to the present art of research in rational difference equations in understanding the behaviour in the complex domain.
\end{abstract}

\vfill
\begin{flushleft}\footnotesize
{Keywords: Pielou's Difference equation, Delay Logistic equation, Chaotic, Local asymptotic stability and Periodicity. }
\end{flushleft}

\section{Introduction and Preliminaries}
Consider the delay logistic difference equation

\begin{equation}
z_{n+1}=\frac{\alpha  z_{n}}{1+ \beta z_{n-1}},  n=0,1,\ldots
\label{equation:total-equationA}
\end{equation}%
where the parameter $\alpha$ is complex number and the initial conditions $z_{-1}$ and $z_{0}$ are arbitrary complex numbers.

\addvspace{\bigskipamount}

The same difference equation including some variations of the this are studied when the parameter $\alpha$ and initial conditions are non-negative real numbers, Eq.(\ref{equation:total-equationA}) was investigated in \cite{B-J} and \cite{C-I}. In this present article it is an attempt to understand the same in the complex plane.

\addvspace{\bigskipamount}

The set of initial conditions ${z_{-1}, ~z_{0}} \in \mathbb{C}$ for which the solution of Eq.(\ref{equation:total-equationA}) is well defined for all $n \geq 0$ is called the $\emph{good}$ set of initial conditions or the \emph{domain of definition}. It is the compliment of the $\emph{forbidden}$ set of Eq.(\ref{equation:total-equationA}) for which the solution is not well defined for some $n \geq 0$. It is really hard to find out either the good set or forbidden set for the second and higher order rational difference equations due to the lack of an explicit form of the solutions. For the rest of the sequel it is assumed that the initial conditions belong to the $\emph{good}$ set \cite{S-E}.

\addvspace{\bigskipamount}
Our goal is to investigate the character of the solutions of Eq.(\ref{equation:total-equationA}) when the parameters are complex and the initial conditions are arbitrary complex numbers in the domain of definition.\\
\noindent
Here, we review some results which will be useful in our investigation of the behavior of solutions of the difference equation (\ref{equation:total-equationA}).

\bigskip
\noindent
Let $f:\mathbb{D}^{2} \rightarrow \mathbb{D}$ where $\mathbb{D}\subseteq \mathbf{C}$ be a continuously differentiable function. Then for any pair of initial conditions $z_0, z_{-1} \in \mathbb{D}$, the difference equation
\begin{equation}
\label{equation:introduction}
z_{n+1} = f(z_{n}, z_{n-1}) \hspace{.25in}
, \hspace{.25in}
\end{equation}
with initial conditions $z_{-1}, z_{0} \in \mathbb{D}.$\\ \\
\noindent
Then for any \emph{initial value}, the difference equation (1) will have a unique solution $\{z_n\}_n$.\\ \\
\noindent
A point $\overline{z}$ $\in$ $\mathbb{D}$ is called {\it \textbf{equilibrium point}} of Eq.(\ref{equation:introduction}) if
$$f(\overline{z}, \overline{z}) = \overline{z}.$$
\noindent
The \emph{linearized equation} of Eq.(\ref{equation:introduction}) about the equilibrium $\bar{z}$ is the linear difference equation

\begin{equation}
\label{equation:linearized-equation}
\displaystyle{
z_{n+1} = a_{0} z_{n} + a_{1}z_{n-1} \hspace{.25in} , \hspace{.25in} n=0,1,\ldots
}
\end{equation}
where for $i = 0$ and $1$.
$$a_{i} = \frac{\partial f}{\partial u_{i}}(\overline{z}, \overline{z}).$$
The \emph{characteristic equation} of
Eq.(\ref{equation:linearized-equation}) is the equation

\begin{equation}
\label{equation:characteristic-roots}
\lambda^{2} - a_{0}\lambda -a_{1} = 0.
\end{equation}

\bigskip
\noindent
The following are the briefings of the linearized stability criterions which are useful in determining the local stability character of the equilibrium
$\overline{z}$ of Eq.(\ref{equation:introduction}), \cite{S-H}.

\bigskip
\noindent
Let $\overline{z}$ be an equilibrium of the difference equation $z_{n+1}=f(z_n,z_{n-1})$.
\begin{itemize}
\item The equilibrium $\bar{z}$ of Eq. (2) is called \textbf{locally stable} if for every $\epsilon >0$, there exists a $\delta>0$ such that for every $z_0$ and $z_{-1}$ $\in$ $\mathbb{C}$ with $|z_0-\bar{z}|+|z_{-1}-\bar{z}|<\delta$ we have $|z_n-\bar{z}|<\epsilon$ for all $n>-1$.
\item The equilibrium $\bar{z}$ of Eq. (2) is called \textbf{locally stable} if it is locally stable and if there exist a $\gamma>0$ such that for every $z_0$ and $z_{-1}$ $\in$ $\mathbb{C}$ with $|z_0-\bar{z}|+|z_{-1}-\bar{z}|<\gamma$ we have $\lim_{n \to\infty}z_n=\bar{z}$.
\item The equilibrium $\bar{z}$ of Eq. (2) is called \textbf{global attractor} if for every  $z_0$ and $z_{-1}$ $\in$ $\mathbb{C}$, we have $\lim_{n \to\infty}z_n=\bar{z}$.
\item The equilibrium of equation Eq. (2) is called \textbf{globally asymptotically stable/fit} is stable and is a global attractor.
\item The equilibrium $\bar{z}$ of Eq. (2) is called \textbf{unstable} if it is not stable.
\item The equilibrium $\bar{z}$ of Eq. (2) is called \textbf{source or repeller} if there exists $r>0$ such that for every $z_0$ and $z_{-1}$ $\in$ $\mathbb{C}$ with $|z_0-\bar{z}|+|z_{-1}-\bar{z}|<r$ we have $|z_n-\bar{z}|\geq r$. Clearly a source is an unstable equilibrium.
\end{itemize}

\section{Local Stability of the Equilibriums and Boundedness}

\label{section:positive-equilibrium} In this section we establish the local stability character of the equilibria of Eq.(\ref{equation:total-equationA}) when the parameters $\alpha$ and $\beta$ are considered to be complex numbers with the initial conditions $z_0$ and $z_{-1}$ are arbitrary complex numbers.

\addvspace{\bigskipamount}

\noindent The equilibrium points of Eq.(\ref{equation:total-equationA}) are
the solutions of the equation
\[
\bar{z}=\frac{\alpha \bar{z}}{1+\beta \bar{z}}
\]
Eq.(\ref{equation:total-equationA}) has two equilibria points
$ \bar{z}_{1,2} = 0, \frac{\alpha-1}{\beta}$ respectively.
\noindent
The linearized equation of the rational difference equation(\ref{equation:total-equationA}) with respect to the equilibrium point $ \bar{z}_{1} = 0$ is

\begin{equation}
\label{equation:linearized-equation}
\displaystyle{
z_{n+1} = \alpha z_{n},  n=0,1,\ldots
}
\end{equation}

\noindent
with associated characteristic equation
\begin{equation}
\lambda^{2} - \alpha \lambda = 0.
\end{equation}

\addvspace{\bigskipamount}

\noindent
The following result gives the local asymptotic stability of the equilibrium $\bar{z}_{1}$.

\begin{theorem}
\label{Result:positive-local-stability1} The equilibriums $\bar{z}_{1}=0$ of Eq.(\ref{equation:total-equationA}) is \\ \\ locally asymptotically stable if and only if $$\abs{\alpha} <1$$\\ \\ unstable if and only if $$\abs{\alpha} < 1$$ \\ \\ and non-hyperbolic if and only if $$ \abs{\alpha} = 1$$
\end{theorem}

\addvspace{\bigskipamount}

\begin{proof}
The characteristic equation of the equilibrium as already mentioned above is $\lambda^{2} - \alpha  \lambda = 0$. The zeros of this polynomials are $0$ and $\alpha$. Therefore the result follows trivially by \emph{Local Stability Theorem}.

\end{proof}

\begin{theorem}
\label{Result:positive-local-stability2} The equilibriums $\bar{z}_{2}$ $=\frac{\alpha-1}{\beta}$ of Eq.(\ref{equation:total-equationA}) is \\ \\ locally asymptotically stable if and only if $$\frac{1}{3} \leq \abs{\alpha} \leq \frac{4}{3}$$\\ \\ unstable if and only if $$ 0< \abs{\alpha} <\frac{1}{3}$$
\end{theorem}

\begin{proof}
The characteristic equation of the equilibrium $\bar{z}_{2}$ is $\lambda^{2} - \lambda + \frac{\alpha-1}{\alpha} = 0$. The zeros of this polynomials are
$\frac{\alpha -\sqrt{(4-3 \alpha ) \alpha }}{2 \alpha }$ and $\frac{\alpha +\sqrt{(4-3 \alpha ) \alpha }}{2 \alpha }$. The equilibrium $\bar{z}_{2}$ is locally asymptotically stable if and only if the modulus of the zeros of the characteristic equation are less than $1$. Simple algebraic calculation of these two inequalities reduced to the condition $\frac{1}{3} \leq \abs{\alpha} \leq \frac{4}{3}$. In the similar fashion the equilibrium is unstable if and only if the modulus of the zeros are greater than $1$. Consequently, the condition reduces to $ 0< \abs{\alpha} <\frac{1}{3}$. Hence the theorem is proved.

\end{proof}

%
%
%\noindent
%Here are few plots of the trajectories for different values of $\alpha$. The plots in the \emph{Fig.1} are for $\alpha=1.9$, $i$ and $3$ and for arbitrary complex initial values $z_0$ and $z_{-1}$. Different color represents different trajectories and here $500$ such trajectory with $5000$ iterations are plotted
%
%\begin{figure}[H]
%      \centering
%
%      \resizebox{15cm}{!}
%      {
%      \begin{tabular}{c c c}
%      \includegraphics [scale=5.0]{19-plot500-1.png}
%      \includegraphics [scale=5.0]{i-plot500-1.png}
%      \includegraphics [scale=5.0]{3-plot500-1.png}
%      \end{tabular}
%      }
%\caption{Trajectory plots of $\abs{\alpha}=1.9$, $1$ and $3$ (Left, Middle and Right respectively).}
%            \end{figure}
%

\noindent
Now we would like to try to find open ball $B(0, \epsilon) \in \mathbb{C}$ such that if $z_n \in B(0, \epsilon)$ and $z_{n-1} \in B(0, \epsilon)$ then $z_{n+1} \in B(0, \epsilon)$ for all $n \geq 0$. In other words, if the initial values $z_0$ and $z_{-1}$ belong to $B(0, \epsilon)$ then the solution generated by the difference equation (\ref{equation:total-equationA}) would essentially be within the open ball $B(0, \epsilon)$.
\begin{theorem}
\label{Result:positive-local-stability2} For every $\epsilon >0$ and for any complex parameter $\abs{\alpha}<1$ and $\abs{\beta}<1$, if $z_n$ and $z_{n-1}$ $\in B(0, \epsilon)$ then $z_{n+1} \in B(0, \epsilon)$ provided, $$0< \epsilon \leq \frac{1-\abs{\alpha}}{\abs{\beta}}$$
\end{theorem}

\noindent

\begin{proof}
Let $\{z_{n}\}$ be a solution of the equation Eq.(\ref{equation:total-equationA}. Let $\epsilon>0$ be any arbitrary real number. Consider $z_n, z_{n-1} \in B(0, \epsilon)$. We need to find out an $\epsilon$ such that $z_{n+1}\in B(0, \epsilon)$ for all $n$. It is follows from the Eq.(\ref{equation:total-equationA}) that for any $\epsilon >0$, using Triangle inequality for $\abs{\alpha}<1$ and $\abs{\beta}<1$ $$\abs{z_{n+1}}=\abs{\frac{\alpha z_n}{1+\beta z_{n-1}}} \leq \frac{\abs{\alpha} \epsilon}{1-\abs{\beta} \epsilon}$$ In order to ensure that $\abs{z_{n+1}}<\epsilon$, it is needed to be $$\frac{\abs{\alpha} \epsilon}{1-\abs{\beta}\epsilon} < 1 \Leftrightarrow \abs{\alpha}-1 < -\abs{\beta} \epsilon \Rightarrow \epsilon < \frac{1-\abs{\alpha}}{\abs{\beta}}$$
\noindent
Therefore the required is followed.

\end{proof}

\section{Periodic Solutions}

\label{section:periodicity}

The global periodicity and the existence of solutions that converge to periodic solutions of the difference equation (\ref{equation:total-equationA} are adumbrated in this section.\\

A solution $\{z_n\}_n$ of a difference equation is said to be \emph{globally periodic} of period $t$ if $z_{n+t}=z_n$ for any given initial conditions. solution $\{z_n\}_n$
is said to be \emph{periodic with prime period} $p$ if p is the smallest positive integer having this property.

\subsection{Existence and Local Stability of Prime Period Two Solutions}
Here we find out prime period two solutions followed by the local stability analysis of them.
\subsubsection{Existence of Prime Period Two Solutions}
Let $\ldots, \phi ,~\psi , ~\phi , ~\psi ,\ldots$, $\phi  \neq \psi $ be a prime period two solution of the difference equation (\ref{equation:total-equationA}). Then $\phi=\frac{\alpha \psi}{1+\beta \phi}$ and $\psi=\frac{\alpha \phi}{1+\beta \psi}$. This two equations lead to the set of solutions (prime period two) except the equilibriums $0$ and $\frac{\alpha-1}{\beta}$ as $\left\{\phi \to \frac{(-0.5-0.5 \alpha ) \beta -0.5 \sqrt{(1+(-2-3 \alpha ) \alpha ) \beta ^2}}{\beta ^2},\psi \to \frac{(-0.5-0.5 \alpha ) \beta +0.5 \sqrt{(1 +(-2 -3 \alpha ) \alpha ) \beta ^2}}{\beta ^2}\right\}$ and $\left\{\phi \to \frac{(-0.5-0.5 \alpha ) \beta +0.5 \sqrt{(1+(-2 - 3 \alpha ) \alpha ) \beta ^2}}{\beta ^2},\psi \to \frac{(-0.5-0.5 \alpha ) \beta -0.5 \sqrt{(1 +(-2 -3 \alpha ) \alpha ) \beta ^2}}{\beta ^2}\right\}$.

\subsubsection{Local Stability of Prime Period Two Solutions}

\noindent
Let $\ldots, \phi ,~\psi , ~\phi , ~\psi ,\ldots$, $\phi  \neq \psi $ be a prime period two solution of the Pielou's equation (\ref{equation:total-equationA}). We set $$u_n=z_{n-1}$$ $$v_n=z_{n}$$
\noindent
Then the equivalent form of the delay Logistic difference equation (\ref{equation:total-equationA}) is
$$u_{n+1}=v_n$$ $$v_{n+1}=\frac{\alpha v_n}{1+ \beta u_{n}}$$
\noindent
Let T be the map on $(0,\infty)\times(0,\infty)$ to itself defined by $$T\left(
                                                                           \begin{array}{c}
                                                                             u \\
                                                                             v \\
                                                                           \end{array}
                                                                         \right)==\left(
                                                                                    \begin{array}{c}
                                                                                      v \\
                                                                                      \frac{\alpha v}{1+ \beta u} \\
                                                                                    \end{array}
                                                                                  \right)
                                                                         $$
%(u,v)=(v, \frac{pu+v}{q+v})$$
\noindent
Then $\left(
                                                                           \begin{array}{c}
                                                                             \phi \\
                                                                             \psi \\
                                                                           \end{array}
                                                                         \right)$ is a fixed point of $T^2$, the second iterate of $T$. \\

$$T^2\left(
                                                                           \begin{array}{c}
                                                                             u \\
                                                                             v \\
                                                                           \end{array}
                                                                         \right)==\left(
                                                                                    \begin{array}{c}
                                                                                      \frac{\alpha v}{1+ \beta u} \\\\
                                                                                      \frac{\alpha \frac{\alpha v}{1+\beta u}}{1+\beta v} \\
                                                                                    \end{array}
                                                                                  \right)
                                                                         $$

$$T^2\left(
                                                                           \begin{array}{c}
                                                                             u \\
                                                                             v \\
                                                                           \end{array}
                                                                         \right)==\left(
                                                                                    \begin{array}{c}
                                                                                      g(u,v) \\
                                                                                      h(u,v) \\
                                                                                    \end{array}
                                                                                  \right)
                                                                         $$

\noindent
 where $g(u,v)=\frac{\alpha v}{1+\beta u}$ and $h(u,v)=\frac{\alpha \frac{\alpha v}{1+\beta u}}{1+\beta v}$. Clearly the two cycle is locally asymptotically stable when the eigenvalues of the
Jacobian matrix $J_{T^2}$, evaluated at $\left(
                                                                           \begin{array}{c}
                                                                             \phi \\
                                                                             \psi \\
                                                                           \end{array}
                                                                         \right)$ lie inside the unit disk.
We have, $$J_{T^2}\left(
                                                                           \begin{array}{c}
                                                                             \phi \\
                                                                             \psi \\
                                                                           \end{array}
                                                                         \right)=\left(
                                                                                   \begin{array}{cc}
                                                                                     \frac{\delta g}{\delta u }(\phi, \psi) & \frac{\delta g}{\delta v}(\phi, \psi) \\\\
                                                                                     \frac{\delta h}{\delta u }(\phi, \psi) & \frac{\delta h}{\delta v }(\phi, \psi) \\
                                                                                   \end{array}
                                                                                 \right)
                                                                         $$

\noindent
where $\frac{\delta g}{\delta u }(\phi, \psi)=-\frac{\alpha  \beta  \psi }{(1+\beta  \phi )^2}$ and $\frac{\delta g}{\delta v }(\phi, \psi)=\frac{\alpha }{1+\beta  \phi }$\\

$\frac{\delta h}{\delta u }(\phi, \psi)=-\frac{\alpha ^2 \beta  \psi }{(1+\beta  \phi )^2 (1+\beta  \psi )}$
and $\frac{\delta h}{\delta v }(\phi, \psi)=-\frac{\alpha ^2 \beta  \psi }{(1+\beta  \phi ) (1+\beta  \psi )^2}+\frac{\alpha ^2}{(1+\beta  \phi ) (1+\beta  \psi )}$
\\ \\
\noindent
Now, set $$\chi=\frac{\delta g}{\delta u }(\phi, \psi)+\frac{\delta h}{\delta v }(\phi, \psi)=\frac{\alpha  \left(-\beta  \psi +\frac{\alpha  (1+\beta  \phi )}{(1+\beta  \psi )^2}\right)}{(1+\beta  \phi )^2}$$ $$\lambda=\frac{\delta g}{\delta u }(\phi, \psi) \frac{\delta h}{\delta v }(\phi, \psi)-\frac{\delta g}{\delta v }(\phi, \psi) \frac{\delta h}{\delta u }(\phi, \psi)=\frac{\alpha ^3 \beta ^2 \psi ^2}{(1+\beta  \phi )^3 (1+\beta  \psi )^2}$$

\noindent
In particular for the prime period $2$ solution, \\ \\ $\left\{\phi \to \frac{(-0.5-0.5 \alpha ) \beta -0.5 \sqrt{(1+(-2-3 \alpha ) \alpha ) \beta ^2}}{\beta ^2},\psi \to \frac{(-0.5-0.5 \alpha ) \beta +0.5 \sqrt{(1 +(-2 -3 \alpha ) \alpha ) \beta ^2}}{\beta ^2}\right\}$,

$$\chi=\frac{\alpha  \left(-\frac{(-0.5-0.5 \alpha ) \beta +0.5 \sqrt{(1 +(-2 -3 \alpha ) \alpha ) \beta ^2}}{\beta }+\frac{\alpha  \left(1+\frac{(-0.5-0.5 \alpha ) \beta -0.5 \sqrt{(1+(-2.-3. \alpha ) \alpha ) \beta ^2}}{\beta }\right)}{\left(1+\frac{(-0.5-0.5 \alpha ) \beta +0.5 \sqrt{(1 +(-2-3 \alpha ) \alpha ) \beta ^2}}{\beta }\right)^2}\right)}{\left(1+\frac{(-0.5-0.5 \alpha ) \beta -0.5 \sqrt{(1+(-2-3 \alpha ) \alpha ) \beta ^2}}{\beta }\right)^2}$$

and

$$\lambda=\frac{\alpha ^3 \left((-0.5-0.5 \alpha ) \beta +0.5 \sqrt{(1 +(-2 -3 \alpha ) \alpha ) \beta ^2}\right)^2}{\beta ^2 \left(1+\frac{(-0.5-0.5 \alpha ) \beta -0.5 \sqrt{(1+(-2-3 \alpha ) \alpha ) \beta ^2}}{\beta }\right)^3 \left(1+\frac{(-0.5-0.5 \alpha ) \beta +0.5 \sqrt{(1 +(-2-3 \alpha ) \alpha ) \beta ^2}}{\beta }\right)^2}$$
\\
\noindent
and for the prime period $2$ solution, \\ \\ $\left\{\phi \to \frac{(-0.5-0.5 \alpha ) \beta +0.5 \sqrt{(1+(-2 - 3 \alpha ) \alpha ) \beta ^2}}{\beta ^2},\psi \to \frac{(-0.5-0.5 \alpha ) \beta -0.5 \sqrt{(1 +(-2 -3 \alpha ) \alpha ) \beta ^2}}{\beta ^2}\right\}$

$$\chi=\frac{\alpha  \left(-\frac{(-0.5-0.5 \alpha ) \beta -0.5 \sqrt{(1 +(-2-3 \alpha ) \alpha ) \beta ^2}}{\beta }+\frac{\alpha  \left(1+\frac{(-0.5-0.5 \alpha ) \beta +0.5 \sqrt{(1+(-2.-3. \alpha ) \alpha ) \beta ^2}}{\beta }\right)}{\left(1+\frac{(-0.5-0.5 \alpha ) \beta -0.5 \sqrt{(1 +(-2-3 \alpha ) \alpha ) \beta ^2}}{\beta }\right)^2}\right)}{\left(1+\frac{(-0.5-0.5 \alpha ) \beta +0.5 \sqrt{(1+(-2-3 \alpha ) \alpha ) \beta ^2}}{\beta }\right)^2}$$
and
$$\lambda=\frac{\alpha ^3 \left((-0.5-0.5 \alpha ) \beta -0.5 \sqrt{(1 +(-2-3 \alpha ) \alpha ) \beta ^2}\right)^2}{\beta ^2 \left(1+\frac{(-0.5-0.5 \alpha ) \beta +0.5 \sqrt{(1+(-2.-3. \alpha ) \alpha ) \beta ^2}}{\beta }\right)^3 \left(1+\frac{(-0.5-0.5 \alpha ) \beta -0.5 \sqrt{(1 +(-2 -3 \alpha ) \alpha ) \beta ^2}}{\beta }\right)^2}$$

\noindent \\
Then it follows from the \emph{Linearized Stability Theorem} that both the eigenvalues of the $J_{T^2}\left(
                                                                           \begin{array}{c}
                                                                             \phi \\
                                                                             \psi \\
                                                                           \end{array}
                                                                         \right)$ lie inside the unit disk if and only if $|\chi| < 1+ |\lambda| < 2$.\\

\noindent
In particular, for $\alpha=i$ and $\beta=2+3i$ the prime period $2$ solutions is $\left\{\phi \to -0.294567+0.313317 i, \psi \to -0.0900486-0.236394 i\right\}$. The periodic trajectory is shown in \emph{Fig.1} of which the $\abs{\chi}=1.38112$ and $\abs{\lambda}=0.689678$. By the Linear Stability theorem ($|\chi| < 1+ |\lambda| < 2$) the prime period $2$ solution is \emph{locally asymptotically stable}.

\begin{figure}[H]
      \centering

      \resizebox{15cm}{!}
      {
      \begin{tabular}{c}
      \includegraphics [scale=5.0]{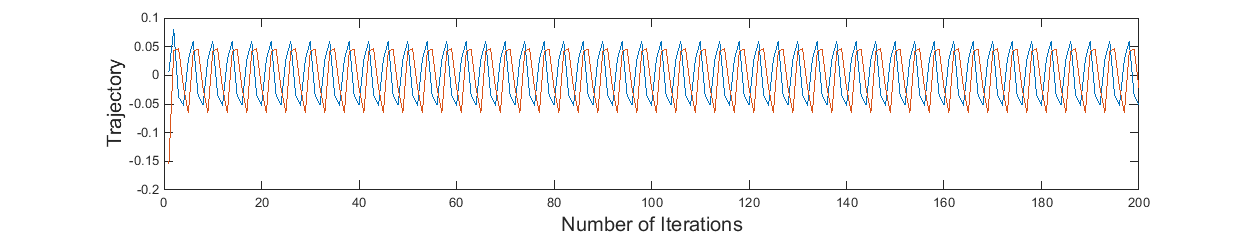}\\\\

      \end{tabular}
      }
\caption{Periodic Cycle solution of period $3$ for $\alpha=i$ and $\beta=2+3i$.}
      \end{figure}

\noindent
For $\alpha=1+i$ and $\beta=2+3i$ the prime period $2$ solutions is $\phi \to -0.168166+0.534411 i,$ \\ $\psi \to -0.370295-0.226718 i$. The periodic trajectory is shown in \emph{Fig.2} of which the $\abs{\chi}=1.5$ and $\abs{\lambda}=1.58114$. Hence the prime period $2$ solution is \emph{unstable}.

\begin{figure}[H]
      \centering

      \resizebox{15cm}{!}
      {
      \begin{tabular}{c}
      \includegraphics [scale=5.0]{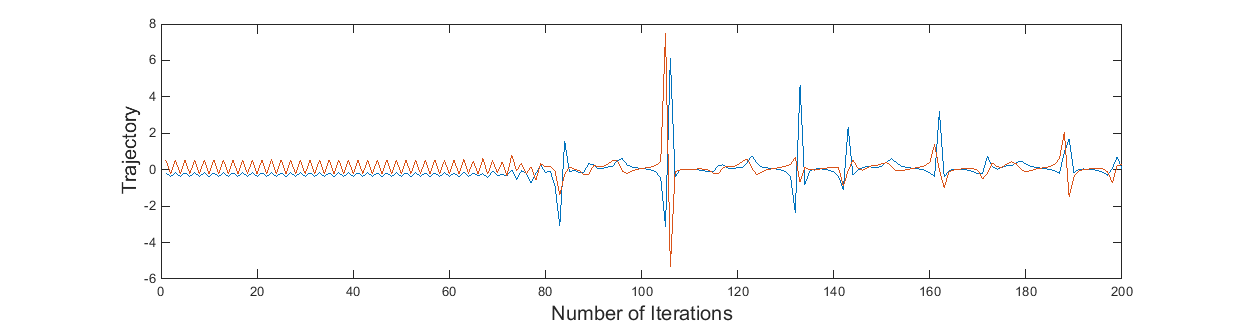}\\\\

      \end{tabular}
      }
\caption{Periodic Cycle solution of period $3$ for $\alpha=i$ and $\beta=2+3i$.}
      \end{figure}

\subsection{Higher Order Cycles in Solutions}
Here we shall explore the higher order periodic cycle of the difference equation (\ref{equation:total-equationA}) in a vivid manner.
\noindent\\
For the parameters $\alpha=i$ and $\beta=2+3i$ of the difference equation, the period $3$ cycle is the following:\\
$\left\{z_0\to 0.316268 +0.129975 i,z_1\to -0.288941+0.157085 i,z_2\to -0.181173-0.056291 i\right\}$. The corresponding periodic trajectory is shown \emph{Fig.3} \\

\begin{figure}[H]
      \centering

      \resizebox{15cm}{!}
      {
      \begin{tabular}{c}
      \includegraphics [scale=5.0]{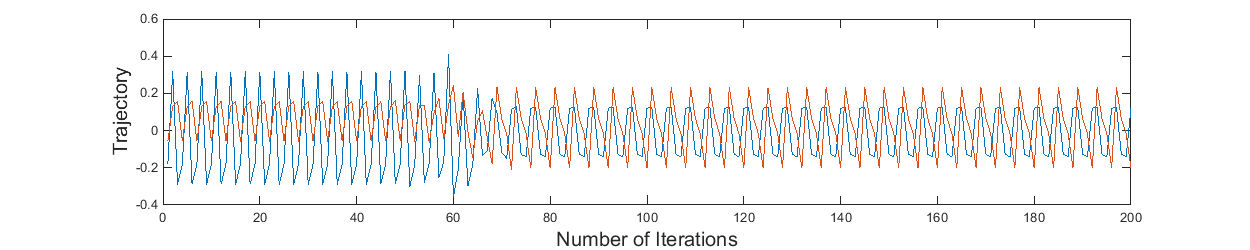}\\\\

      \end{tabular}
      }
\caption{Periodic Cycle solution of period $3$ for $\alpha=i$ and $\beta=2+3i$.}
      \end{figure}

\noindent
For the parameter $\alpha=\frac{1}{3}+i$ and $\beta=2+i$ a periodic cycle of period $10$ has been encountered. The periodic trajectory is \\ $z_0\to -0.0197446-1.28723 i,z_1\to 1.03398 +0.925847 i,z_2\to -0.406487+0.128166 i,$ \\ $z_3\to -0.125003-0.00142325 i,z_4\to 0.63328 -0.516259 i,z_5\to 0.83925 +0.756558 i,z_6\to -0.223017+0.36021 i,z_7\to -0.116754+0.0893373 i,z_8\to -0.239031+0.16474 i,z_9\to -0.382611-0.236912 i$. The plot of the trajectory is given in \emph{Fig.4}.

\begin{figure}[H]
      \centering

      \resizebox{15cm}{!}
      {
      \begin{tabular}{c}
      \includegraphics [scale=5.0]{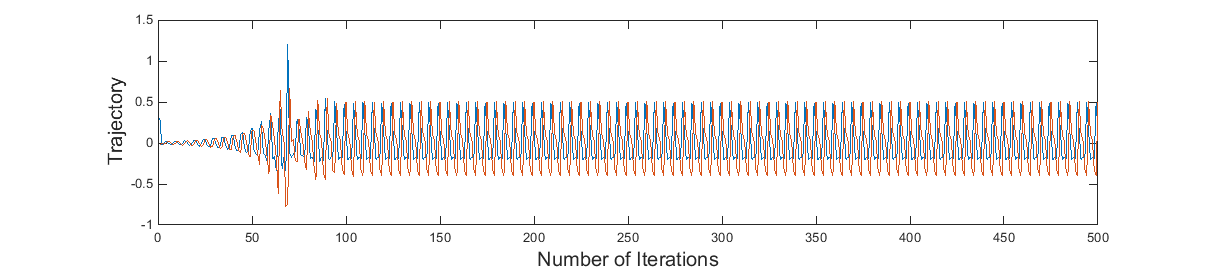}\\\\

      \end{tabular}
      }
\caption{Periodic Cycle solution of period $3$ for $\alpha=i$ and $\beta=2+3i$.}
      \end{figure}

\noindent
Keeping fixed $\beta$ as $1$ computationally it is also seen that for $\alpha=(15, 26)$ and $(55, 95)$, almost all solutions corresponding to the initial values $z_0, z_{-1} \in \mathbb{D}$ converges to the periodic point $(-356.366, -194.0009)$ and (11.6656, -0.1928) of period $55$ and $199$ respectively. Certainly there are many more such examples of $\alpha$ for which the same happens.
\noindent
The plot of the periodic trajectory are given the \emph{Fig.5} and \emph{Fig.6} for the above two examples.

\begin{figure}[H]
      \centering

      \resizebox{12cm}{!}
      {
      \begin{tabular}{c}
      \includegraphics [scale=5.0]{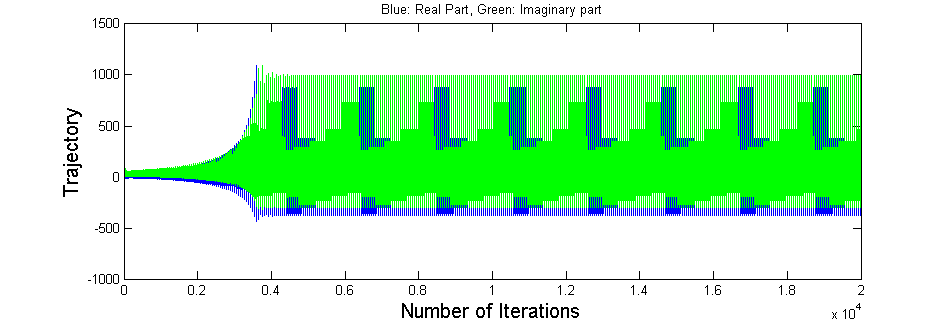}\\\\
      \includegraphics [scale=4.0]{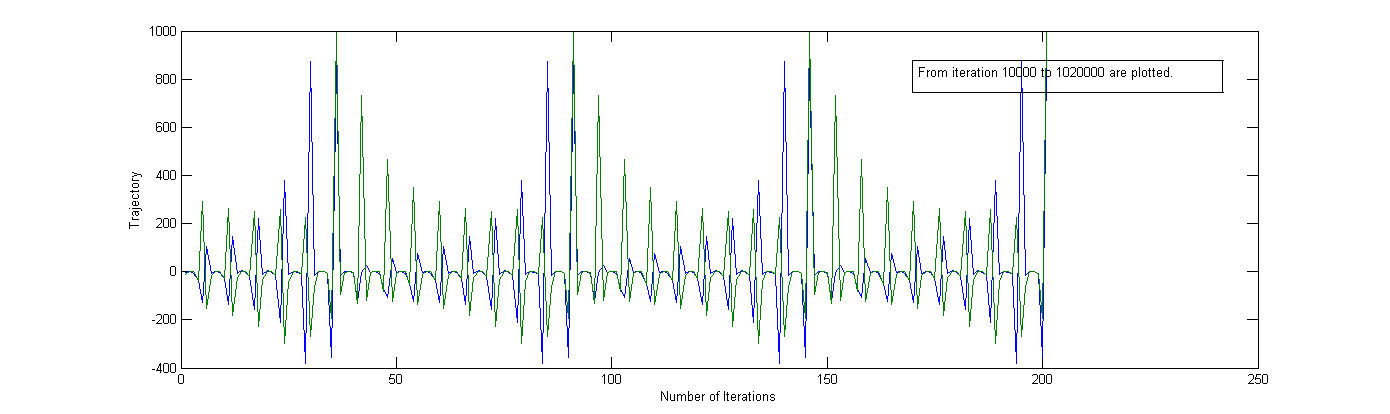}\\

      \end{tabular}
      }
\caption{Periodic Cycle solution for $\alpha=(15, 26)$ and $\beta=1$.}
      \begin{center}

      \end{center}
      \end{figure}

\noindent

\begin{figure}[H]
      \centering

      \resizebox{12cm}{!}
      {
      \begin{tabular}{c}
      \includegraphics [scale=4.0]{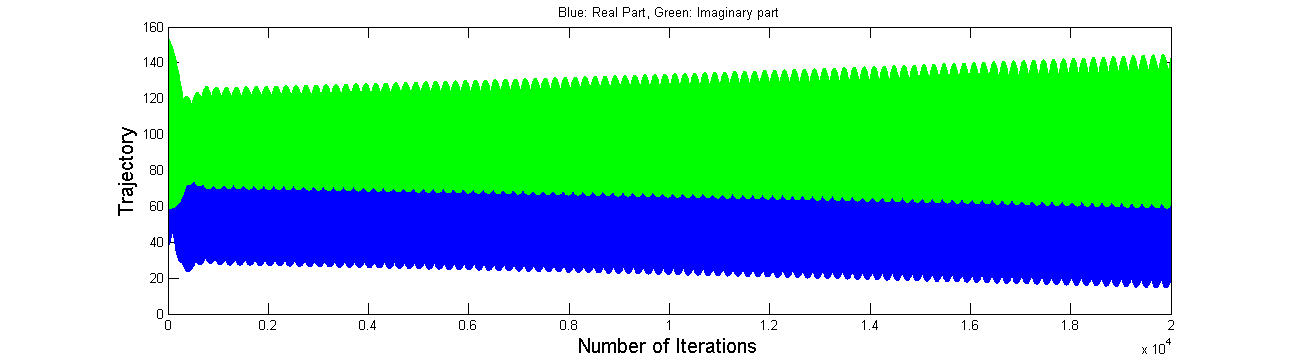}\\\\
      \includegraphics [scale=4.0]{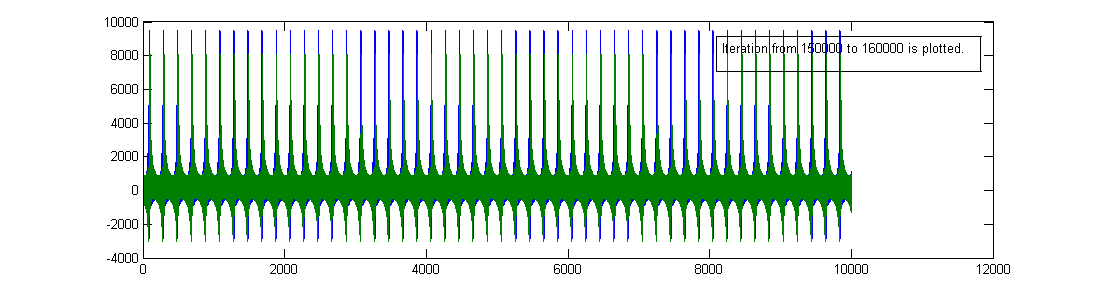}\\

      \end{tabular}
      }
\caption{Periodic Cycle solution for $\alpha=(55, 95)$ and $\beta=1$.}
      \begin{center}

      \end{center}
      \end{figure}

\noindent
In both the \emph{Fig.5} and \emph{Fig.6}, the two figures are shown with different number of iterations of the periodic trajectory.

\noindent
Now we shall demonstrate few computational examples where periodic cycle exists with \emph{very high period} roughly of order $1000$. We choose $\alpha=(35, 94)$ and $\beta=(88,55)$ and for arbitrary complex initial values $z_0$ and $z_{-1}$, the trajectory of periodic cycles are plotted and corresponding phase space also plotted in the \emph{Fig.7}.

\begin{figure}[H]
      \centering

      \resizebox{11cm}{!}
      {
      \begin{tabular}{c}
      \includegraphics [scale=4.0]{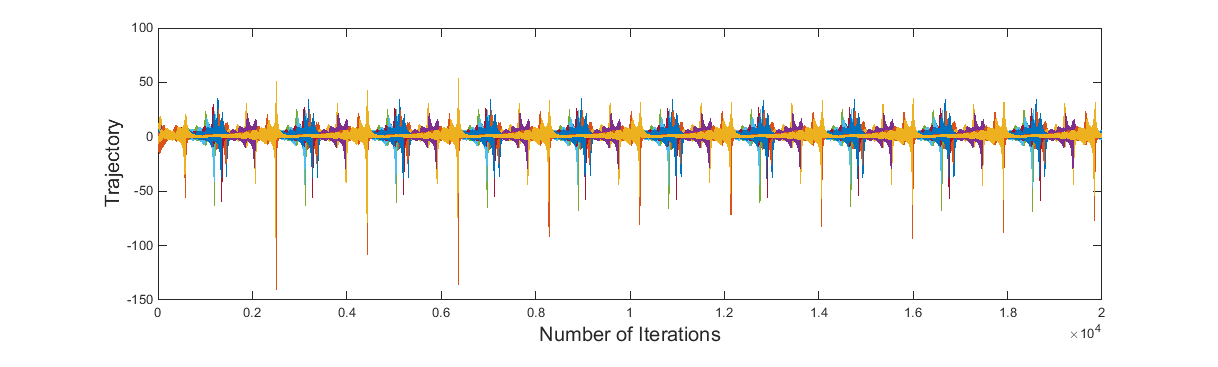}\\
      \includegraphics [scale=4.0]{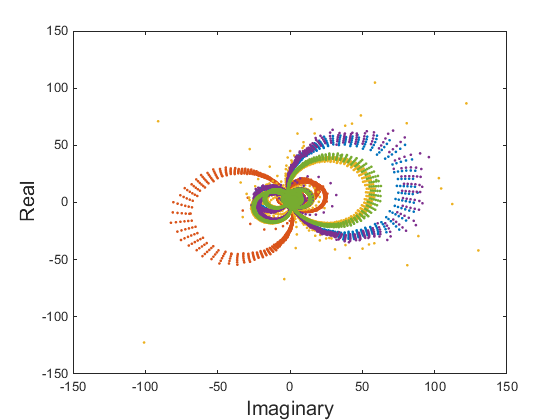}\\
      \includegraphics [scale=4.0]{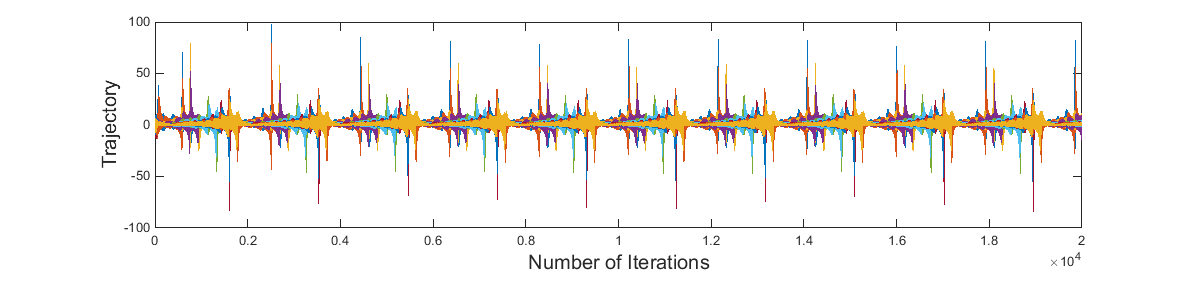}\\
      \includegraphics [scale=4.0]{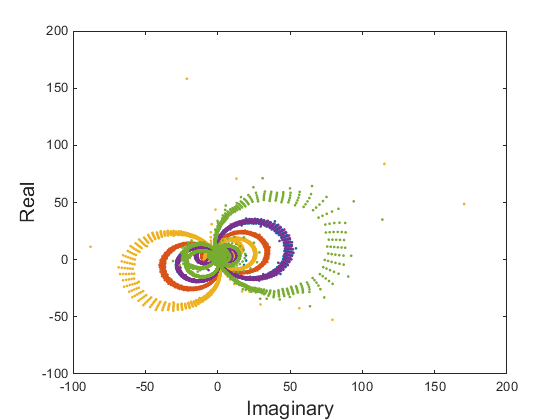}\\
      \includegraphics [scale=4.0]{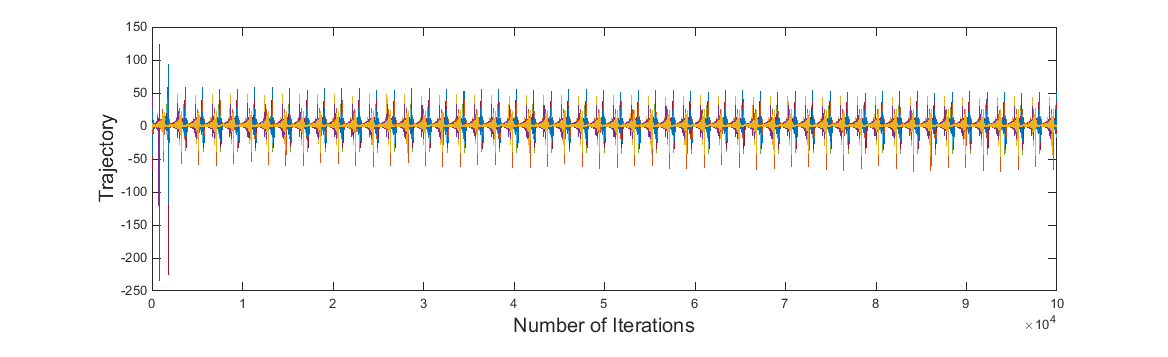}\\
      \includegraphics [scale=4.0]{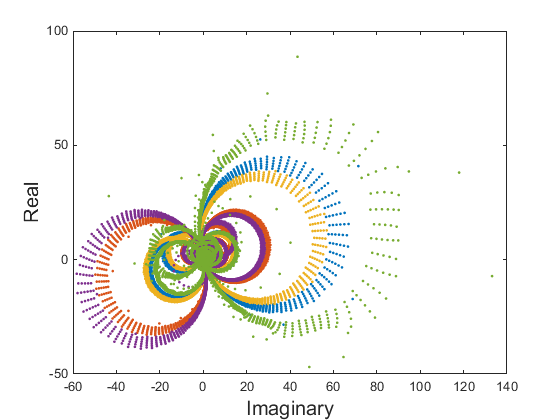}\\

      \end{tabular}
      }
\caption{Periodic Cycle solution for $\alpha=(35, 94)$ and $\beta=(88, 55)$. In Each figure 5 set of initial values are taken and corresponding trajectories are plotted.}
      \begin{center}

      \end{center}
      \end{figure}

\noindent
From the computational aspect, in each of the plots it is evident that for $\alpha=(35, 94)$ and $\beta=(88,55)$ the delay logistic difference equation possess very high order periodic cycles eventually for almost all initial values.

\section{Chaotic Solutions}
Finding chaotic solutions for the delay logistic equation (\ref{equation:total-equationA}) is interesting indeed since in case of real parameter $\alpha$ and $\beta$ and the initial values $z_0$ and $z_{-1}$ there does not exists any chaotic solutions. \\
 \noindent
 The method of Lyapunov characteristic exponents serves as a useful tool to quantify chaos. Specifically Lyapunov exponents measure the rates of convergence or divergence of nearby trajectories. Negative Lyapunov exponents indicate convergence, while positive Lyapunov exponents demonstrate divergence and chaos. The magnitude of the Lyapunov exponent is an indicator of the time scale on which chaotic behavior can be predicted or transients decay for the positive and negative exponent cases respectively. In this present study, the largest Lyapunov exponent is calculated for a given solution of finite length numerically \cite{Wolf}.

 \noindent
 We are looking for complex parameter $\alpha$ and $\beta$ for which for every initial values the solutions are chaotic.  If we consider the parameter $\beta=1$ then the difference equation is known as \emph{Pielou's equation} and which is well studied earlier in case of real line. Here are few examples which we came across computationally.

\begin{table}[H]

\begin{tabular}{| m{7cm} | m{8cm} |}
\hline
\centering   \textbf{Parameter} $\alpha$, $\beta=1$ &
\begin{center}
\textbf{Internal of Lyapunav exponent}
\end{center}\\
\hline
\centering $\alpha=(8, 43)$, $\beta=1$&
\begin{center}
$(1.205, 2.623)$
\end{center}\\
\hline
\centering $\alpha=(1, 97)$, $\beta=1$ &
\begin{center}
$(1.845, 3.028)$
\end{center}\\
\hline
\centering $\alpha=(6, 53)$, $\beta=1$ &
\begin{center}
$(0.785, 1.718)$
\end{center} \\
\hline
\centering $\alpha=(12, 50)$, $\beta=1$ &
\begin{center}
$(0.373, 1.485)$
\end{center}\\
\hline

\end{tabular}
\caption{Cycle solutions of the Pielou's equation ($\beta=1$) for different choice of $\alpha$ and initial values.}
\label{Table:}
\end{table}

\noindent
The largest Lyapunav exponents of the solutions for different initial values are lying in the positive intervals as stated above in the table. This ensures that the solutions are chaotic. It is observed that the chaos is bounded.

\begin{figure}[H]
      \centering

      \resizebox{10cm}{!}
      {
      \begin{tabular}{c}
      \includegraphics [scale=4]{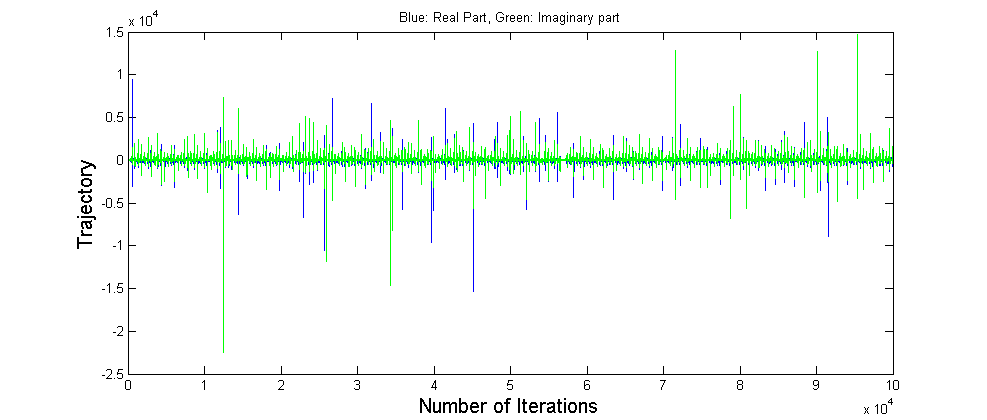}\\
      \includegraphics [scale=4]{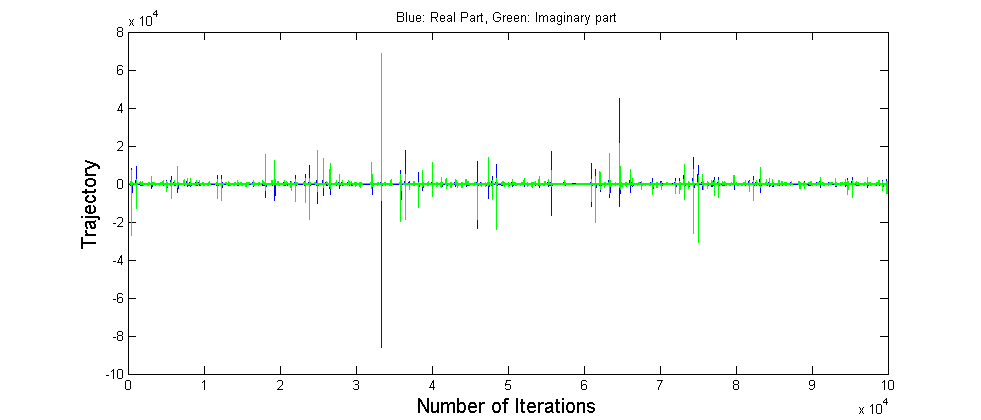}\\
      \includegraphics [scale=4]{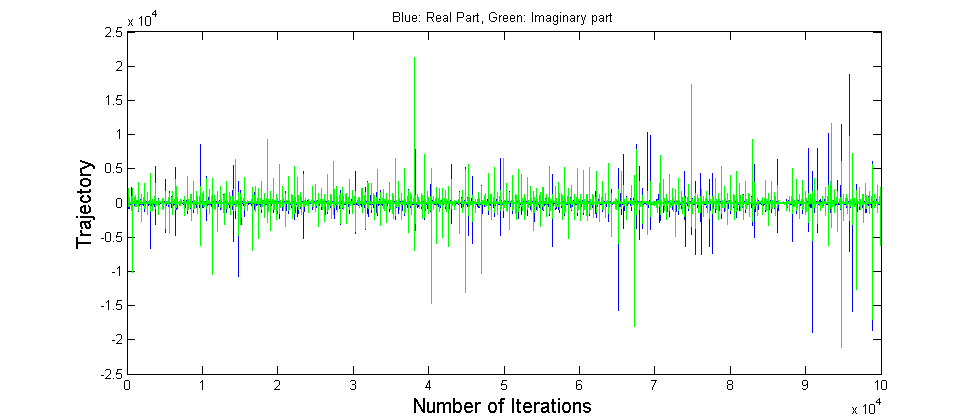}\\
      \includegraphics [scale=4]{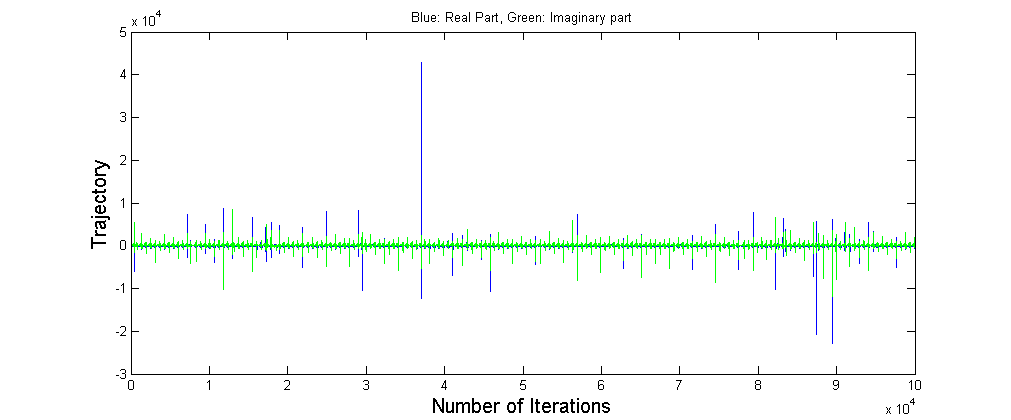}\\

      \end{tabular}
      }
\caption{Chaotic Solutions for the Pielou's equation ($\beta=1$) of four different cases as stated in table 1.}
      \begin{center}

      \end{center}
      \end{figure}
\noindent

\section{Some Interesting Nontrivial Problems}

The computational experiment endorses us to pose the following important open problems in this context.

\begin{openproblem}
Find out the subset of the $\mathbb{D}$ of all possible initial values $z_0$ and $z_1$ for which the solutions of the delay logistic equation possess chaotic solutions for a given parameter $\alpha$ and $\beta$.
\end{openproblem}

\begin{openproblem}
Find out the complex parameters $\alpha$ and $\beta$ such that for any initial values $z_0$ and $z_{-1}$ from the $\mathbb{D}$ the solutions of the delay logistic equation are chaotic.
\end{openproblem}

\begin{openproblem}
Find out the parameters $\alpha$ and $\beta$ such that for any initial values $z_0$ and $z_{-1}$ from the $\mathbb{D}$ the solution of the difference equation are periodic (globally).
\end{openproblem}

\begin{openproblem}
Characterize the parameters $\alpha$ and $\beta$ such that for any initial values $z_0$ and $z_{-1}$ from the $\mathbb{D}$ the solution of the difference equation are periodic (globally). How large the period could be? Is it possible to find an upper bound?
\end{openproblem}

%%%%%%%%%%%%%%%%%%%%%%%%%%%%%%%%%%%%%%%%%%%%%%%%%%%%%%%%
\section{Future Endeavours}

In continuation of the present work for a generalization of the delay logistic equation, $\frac{\alpha z_{n-l}}{1+\beta z_{n-k}}$ with varies $\alpha$ and $\beta$, where $l$ and $k$ are delay terms and it demands similar analysis which we plan to pursue in near future.

%%%%%%%%%%%%%%%%%%%%%%%%%%%%%%%%%%%%%%%%%%%%%%%%%%%%%%%%%%%%%%%%%%%%%%%%%%%%%%%%%%%%%%%%%%%%%%%%%%%%%%%%%%%%%%
%Acknowledgement
%%%%%%%%%%%%%%%%%%%%%%%%%%%%%%%%%%%%%%%%%%%%%%%%%%%%%%%%%%%%%%%%%%%%%%%%%%%%%%%%%%%%%%%%%%%%%%%%%%%%%%%%%%%%%%%%%%
\section*{Acknowledgement}
The author thanks \emph{Dr. Esha Chatterjee} and \emph{Dr. Pallab Basu} for discussions and suggestions.

%%%%%%%%%%%%%%%%%%%%%%%%%%%%%%%%%%%%%%%%%%%%%%%%%%%%%%%%%%%%%%%%%%%%%%%%%%%%%%%%%%%%%%%%%

\end{document}